\documentclass{amsart}

\usepackage{amsmath, amsthm, amssymb}
\usepackage{bbm}
\usepackage[utf8]{inputenc}
\usepackage[T1]{fontenc}
\usepackage{lmodern} 
\usepackage{xcolor}

\newtheorem{thm}{Theorem}[section]
\newtheorem{cor}[thm]{Corollary}
\newtheorem{lem}[thm]{Lemma}
\newtheorem{conj}[thm]{Conjecture}
\newtheorem{rmk}[thm]{Remark}
\newtheorem{lemma}[thm]{Lemma}
\newtheorem{definition}[thm]{Definition}
\newtheorem{example}[thm]{Example}

\newcommand{\zsn}{\mathbb{Z}[S_n]}
\newcommand{\csm}{c_{SM}}
\newcommand{\csmt}{c_{SM}^T}

\newcommand{\fk}{FK_n}

\newcommand{\bfi}{{\bf i}}
\newcommand{\bfj}{{\bf j}}
\newcommand{\hata}{\hat{\mathbb{A}}}

\renewenvironment{proof}{{\it Proof.}}{\qed}

\title{Chern class of Schubert cells in the flag manifold and related algebras}
\author{Seung Jin Lee}
\date{}

\begin{document}
\maketitle
\begin{abstract}
We discuss a relationship between Chern-Schwartz-MacPherson classes for Schubert cells in flag manifolds, Fomin-Kirillov algebra, and the generalized nil-Hecke algebra. We show that nonnegativity conjecture in Fomin-Kirillov algebra implies the nonnegativity of the Chern-Schwartz-MacPherson classes for Schubert cells in flag manifolds for type A. Motivated by this connection, we also prove that the (equivariant) Chern-Schwartz-MacPherson classes for Schubert cells in flag manifolds are certain summations of the structure constants of the equivariant cohomology of the Bott-Samelson varieties. We also discuss the refined positivity conjectures of the Chern-Schwartz-MacPherson classes for Schubert cells motivated by the nonnegativity conjecture in Fomin-Kirillov algebra.
\end{abstract}
\section{introduction}
The Chern classes are characteristic classes associated with vector bundles on a smooth variety. A functorial theory of Chern classes for possibly singular variety was conjectured by Grothendieck and Deligne, and established by MacPherson \cite{Mac74}. For a constructible function $\phi$ on a complex variety $X$, the theory associates a class $c_*(\phi) \in H_*(X)$ such that $c_*(\mathbbm{1}_X)= c(X) \cap [X]$ when $X$ is a smooth compact complex variety. The functoriality properties satisfied by theses classes determine them uniquely. If $X$ is a compact complex variety, then the class $c_*(\mathbbm{1}_X)$ coincides with a class defined by Schwartz \cite{Sch65a,Sch65b}, so the class is commonly called the Chern-Schwartz-MacPherson (CSM) class of $X$.\\

Aluffi and Mihalcea studied the CSM classes of Schubert cells in Grassmannians and flag varieties \cite{AM09,AM15}. For an element $w$ in the Weyl group, the CSM class of the Schubert cell $X(w)^\circ$ is $c_*(\mathbbm{1}_{X(w)^\circ})$, which we denote by $\csm(X(w)^\circ)$. In \cite{AM15}, Aluffi and Mihalcea provided a beautiful reculsive formula to compute the CSM classes of Schubert cells $X(w)^\circ$ in the flag variety by investigating the relationship between $X(w)^\circ$ and $X(ws_k)^\circ$ where $s_k$ is a simple reflection. They also conjectured that the CSM classes of Schubert cells are nonegative, i.e., when we write 
$$\csm(X(w)^\circ)= \sum_u c(w;u) [X(u)]$$ in terms of Schubert basis $[X(u)]$ in the homology of the flag variety, $c(w;u)$ is nonnegative. Huh \cite{Huh} provides a proof of the positivity of the CSM classes of Schubert cells in the Grassmannian, but his technique does not seem to apply in the flag variety. \\

In this paper, we connect the theory of CSM classes of Schubert cells in a flag variety for type A with the seemingly unrelated algebras, called the Fomin-Kirillov algebra and the generalized nil-Hecke algebra. Fomin and Kirillov \cite{FK99} studied a certain quadratic algebra, also called the Fomin-Kirillov algebra, to understand the combinatorics of the cohomology of the flag variety. There is a commutative subalgebra of the Fomin-Kirillov algebra generated by Dunkl elements, which is isomorphic to the cohomology of the flag variety. They conjectured that Schubert polynomials evaluated at Dunkl elements can be written as a nonnegeative linear combination of non-commutative monomials, which is called the nonnegativity conjecture.\\

The first main theorem in this paper is the following:

\begin{thm} \label{main1} The nonnegativity conjecture implies the nonnegativity of the CSM classes of Schubert cells in the flag variety for type A. \end{thm}

 It turns out that the nonnegativity conjecture for the Fomin-Kirillov algebra is much stronger than the positivity conjecture for CSM classes conjectured by Aluffi and Mihalcea for type A. Motivated by this fact, we investigate the relation between the equivariant CSM classes of Schubert cells in the flag variety of all types and the structure constants of the equivariant cohomology of the Bott-Salemson varieties. Ohmoto proves \cite{Ohm} that there is an equivariant version of $c_*$, denoted by $c^T_*$ such that for a constructible set $U$ in a complex variety $X$, $c^T_*(\mathbbm{1}_U)$ is an element in the equivariant homology $H^T_*(X)$ of $X$ that satisfies $c^T_*(\mathbbm{1}_Y)= c^T(T_Y) \cap [Y]_T$ if $Y$ is a projective, non-singular variety. Denote the CSM class of the Schubert cell $c^T_*(\mathbbm{1}_{X(w)^\circ})$ by $\csmt(X(w)^\circ)$. Aluffi and Mihalcea \cite{AM15} shows that the their reculsive formula for $\csmt(X(w)^\circ)$ is essentially the same as those for the ordinary CSM classes (See Section 2 for details). \\
 
Allufi and Mihalcea \cite{AM15} also conjectured that $\csmt(X(w)^\circ)$ also possess positivity. More precisely, when we write
$$\csmt(X(w)^\circ)=\sum_u c_T(w;u) [X(u)]_T,$$
 they conjectured that $c_T(w;v)$ is a nonnegative polynomial in simple roots. The second main theorem is the following.
 
\begin{thm}\label{main2} For elements $w,v$ in the Weyl group, $c_T(w;v)$ is a certain sum of the structure constants of the equivariant Bott-Samelson variety (See Section 6 for the precise statement.)
\end{thm}
This result together with Theorem \ref{main1} suggests interesting conjectures about the positivity of refined sum of the constants of the equivariant Bott-Samelson variety strictly stronger than Aluffi and Mihalcea's conjectures. One of refined conjectures is the following.
 
 \begin{conj}\label{refined} For Weyl group elements $w,u,v$ with $\ell(w)=\ell(u)+\ell(v)$, let $p^w_{u,v}$ be the Littlewood-Richardson coefficients, the structure constants of the cohomology of the flag variety. Then 
 $$c(w;v) \geq \sum_{\substack{u\\ \ell(w)=\ell(u)+\ell(v)}} p^w_{u,v}.$$
 \end{conj}
 
We use the generalized nil-Hecke algebras studied by Berenstein and Richmond \cite{BR15} for studying the equivariant structure constants of the cohomology of the Bott-Samelson varieties and relate the generalized nil-Hecke algebra with the Fomin-Kirillov algebra and the operator $T_k$ studied by Aluffi and Mihalcea (See Section 2 for the definition of $T_k$).\\

 The structure of the paper is as follows: In Section 2, we discuss Allufi and Mihalcea's reculsive formula about the CSM class of the Schubert cells in the flag variety. In Section 3, we introduce the Fomin-Kirillov algebra and describe the nonnegativity conjecture. In Section 4, we state Theorem \ref{main1} precisely without the proof and discuss its implications. In Section 5, we define the generalized nil-Hecke algebra and connect the algebra with the structure constants of the equivariant cohomology of the Bott-Samelson variety. In Section 6, we interpret $T_k$ in terms of the generalized nil-Hecke algebra and show Theorem \ref{main2}. In Section 7, we prove Theorem \ref{main1} by relating Section 6 with the Fomin-Kirillov algebra. In Section 8, we discuss refined positivity conjectures and other concluding remarks.
  
\section{CSM classes for Schubert cells in flag varieties}
 In this section, we give a brief introduction to the CSM class and the summary of results in \cite{AM15} by Allufi and Mihalcea.\\

Let $G$ be a complex simple Lie group, $B$ be a Borel subgroup, and $T$ a maximal torus in $B$. Let $\mathfrak{g},\mathfrak{h}$ be the Lie algebras of $G$ and $T$, and let $R\subset \mathfrak{h}^*$ be the associated root system with the set of positive roots $R^+$ determined by $B$. Denote by $\{\alpha_1,\ldots,\alpha_r\} \subset R^+$ the set of simple roots. Let $R^\vee$ denote the set of coroots $\alpha^\vee \in \mathfrak{h}$ and $\langle \cdot,\cdot\rangle : R\otimes R^\vee \rightarrow \mathbb{Z}$ the pairing.\\

Let $W$ be the Weyl group of $G$ generated by simple reflections $s_1,\ldots,s_r$ corresponding respectively to the simple roots $\alpha_1,\ldots,\alpha_r$ of $G$. Let $X:=G/B$ be the flag variety and $X(w)^\circ:=BwB/B$ the Schbuert cell for $w$ in the flag variety $X$. Each cell $X(w)^\circ$ is isomorphic to $\mathbb{C}^{\ell(w)}$. Let $X(w):= \overline{BwB/B}$ be the Schubert variety for $w$. Each Schubert variety $X(w)$ has a fundamental class $[X(w)] \in H_*(X,\mathbb{Q})$, and these classes form a basis for the (co)homology of the flag variety. It may be verified that $X(v) \subset X(w)$ if and only if $v\leq w$ in the Bruhat order (See \cite{Ful97} for instance). For type $A$ case, we have $G=SL(n)$ and $B$ the set of upper triangular matrices in $G$. The Weyl group $W$ is the symmetric group $S_n$ and the for elements $w,v \in S_n$, the Bruhat order $v \leq w$ is defined by $v=wt_{ij}$ for some reflection $t_{ij} \in S_n$ with $\ell(v)=\ell(w)-1$.\\

Let $\partial_k$ be the BGG operator \cite{BGG73} from $H_*(X) \rightarrow H_{*+2}(X)$ for each $1\leq k \leq r$, which sends the Schubert class $[X(w)]$ to $[X(ws_k)]$ if $\ell(ws_k)>\ell(w)$ and to 0 otherwise. The Weyl group admits a right action on $H_*(X)$, and one can write the action of $s_k$ on $H_*(X)$ in terms of $\partial_k$ and the first Chern class of the line bundle $\mathcal L_{\alpha_k}$ indexed by $\alpha_k$. More precisely, we have $s_k=id-c_1(\mathcal L_{\alpha_k})\partial_k$ where the action of the first Chern class of line bundle $c_1(\mathcal L_\lambda)$ indexed by an integral weight $\lambda$ on $H_*(X)$ is given by the Chevalley formula:
$$c_1(\mathcal L_{\lambda})\cdot [X(w)]= \sum \langle \lambda, \beta^\vee\rangle [X(ws_\beta)]$$
where the sum is over all positive roots $\beta$ such that $\ell(ws_\beta)=\ell(w)-1$. See e.g. \cite{Ful97} for more details. \\

For $1\leq k \leq r$, define the non-homogeneous operator
$$T_k:=\partial_k-s_k : H_*(X)\rightarrow H_*(X).$$

The CSM class of the Schubert cell $X(w)^\circ$ is defined by $c_*(\mathbbm{1}_{X(w)^\circ}) \in H_*(G/B)$, and we denote by $\csm(X(w)^\circ)$. Although one needs to describe functorial properties of $c_*$ \cite{Mac74} to define the CSM class of the Schubert cells, Aluffi and Mihalcea \cite{AM15} provided the following recursive formula for $\csm(X(w)^\circ)$ which is purely combinatorial.

\begin{thm}  \label{rec} \cite{AM15}
 Let $w \in W$ and $1\leq k \leq r$. Then
$$ T_k(\csm(X(w)^\circ))= \csm(X(ws_k)^\circ)).$$
\end{thm}

For $w \in W$, one can write $\csm(X(w)^\circ)=\sum_{v \in W} c(w;v) [X(v)]$. It is not hard to show from Theorem \ref{rec} that $c(w;v)=0$ unless $v \leq w$. Aluffi and Mihalcea proposed the following positivity conjecture.

\begin{conj}\cite[Conjecture 1]{AM15}\label{poscsm}
For $v\leq w$, we have $c(w;v)>0$.
\end{conj}
Note that the $T_k$ action does not preserve the positivity. \\

Let $P\in G$ be a parabolic subgroup containing $B$, and let  $p: G/B \rightarrow G/P$ be the natural projection. Then $p(X(w))=X(w W_p)$, where $W_P$ is the subgroup of $W$ generated by the reflections in $P$. The functoriality of CSM classes can be used to show that
$$p_*(\csm(X(w)^\circ))= \csm (X(wW_p)^\circ).$$
Therefore, Theorem \ref{poscsm} implies the nonnegativity of CSM classes of any $G/P$. When $G/B$ is a Grassmannian manifold, positivity of special cases of CSM classes in $G/P$ are proved by several authors \cite{AM09,Mih15,Jon10,Str11}, and it was settled for all cases by Huh \cite{Huh}.\\

Aluffi and Leonardo also shows that Theorem \ref{rec} works in the equivariant setting. Let $T\subset B$ be the maximal torus in $B$. For a variety $X$ with a $T$-action, the equivariant cohomology $H^*_T(X)$ is the ordinary cohomology of $(ET \times X)/T$, where $ET$ is the universal $T$-bundle. The equivariant cohomology is an algebra over $H^*_T(pt)$, a polynomial ring in generators for the weight lattice of $T$. Since $G/B$ is smooth, we identify the equivariant homology $H^T_*(G/B)$ with the equivariant cohomology $H^*_T(G/B)$. For a subvariety $Y\subset X$ invariant under the $T$ action, one can associate its equivariant fundamental class $[Y]_T \in H^T_*(X)$. See e.g. \cite{Ful07} for basics of the equivariant cohomology. \\

Ohmoto \cite{Ohm} shows that one can define the functor $c_*^T$ from the group of equivariant constructible functions to the equivariant homology $H_*^T(X)$ satisfying certain functorial properties and $c^T_*(\mathbbm{1}_Y)= c^T(T_Y) \cap [Y]_T$ if $Y$ is a projective, non-singular variety. The image of $c^T_*(\mathbbm{1}_Y)$ via the map $H_*^T(X)\rightarrow H_*(X)$ is the CSM class $c_*(\mathbbm{1}_Y)$. Let $i$ be the map $H_*^T(G/B)\rightarrow H_*(G/B)$ induced by the projection $H^T_*(pt) \rightarrow H_*(pt)\cong\mathbb{Q}$ sending all generators in weight lattice to 0. Then we have $i(\csmt(X(w)^\circ))=\csm(X(w)^\circ)$ and $i(c_T(w;v))=c(w;v)$.

\section{Fomin-Kirillov algebra}

Fomin and Kirillov defined a certain quadratic algebra $\fk$, also called the Fomin-Kirillov algebra, to better understand the combinatorics of the cohomology ring of the flag variety. They show that the commutative subalgebra generated by Dunkl elements of degree 1 is isomorphic to the cohomology of the flag variety. Since then, a lot of variations for the quadratic algebra has been studied \cite{Kir15,KM03,KM04,KM05,KM10,KM12}. For example, there are generalizations of the Fomin-Kirillov algebra for K-theory, quantum, equivariant cohomology and for other finite types.\\

We review some facts on the Fomin-Kirillov algebra mostly proved in \cite{FK99}.
\begin{definition}
For a fixed positive integer $n$, let $\fk$ be the the free algebra generated by $\{[ij]: i,j\in \mathbb{Z},1\leq i<j \leq n\}$ with the following relations:
\begin{align*}
[ij]^2&=0.\\
[ij][kl]&=[kl][ij] \text{ for distinct } i, j, k ,l. \\
[ij][jk]&=[jk][ik]+[ik][ij] \text{ and } [jk][ij]=[ik][jk]+[ik][ij] \text{ for distinct } i, j, k.
\end{align*}
\end{definition}
There are two representations of $\fk$ on $\zsn$. Define the \emph{Bruhat action} of $[ij]$ on $\zsn$ by
\begin{equation}\label{bruhat}
w\cdot[ij] =
\begin{cases}
wt_{i,j} & \text{if}\quad \ell(wt_{i,j})=\ell(w)-1 \\
0 & \text{otherwise.}
\end{cases}
\end{equation}
where the $\ell(w)$ is the lenght of $w$ in $S_n$. The Bruhat action is dual to the action described in \cite{FK99}. Note that Fomin and Kirillov \cite{FK99} defined the action such that $w[ij]=wt_{ij}$ if $\ell(wt_{i,j})=\ell(w)+1$ and $0$ otherwise.\\

The following action is less known but closely related to the CSM classes of the Schubert cells in the flag variety. Define the \emph{extended Bruhat action} of $[ij]$ on $\zsn$ by
\begin{equation}\label{ebruhat}
w*[ij] =
\begin{cases}
wt_{i,j} & \text{if}\quad \ell(wt_{i,j})<\ell(w) \\
0 & \text{otherwise.}
\end{cases}
\end{equation}
Note that the Bruhat action is the degree $-1$ part of the extended Bruhat action. The proof that the extended Bruhat action of $[ij]$ provides a representation of $\fk$ on $\zsn$ is straightforward calculation.\\

\begin{rmk} The $*$-action can be interpreted as edges in the Bruhat graph, although author is not aware that $*$-action is used in past literatures. 
\end{rmk}

Let $\theta_i$ be the Dunkl elements defined by $\sum_{j\neq i}[ij]$. Here for $i<j$, $[ji]$ is the same as $-[ij]$. Then $\theta_i$ commutes pairwise for all $i$. Moreover, all symmetric functions in Dunkl elements vanish in $\fk$ and these are all relations between $\theta_i$. Therefore, the commutative subalgebra $\mathcal C$ generated by Dunkl elements is isomorphic to the cohomology of the flag variety. See \cite{FK99} for details.\\

Let $\mathfrak{S}_w$ be the Schubert polynomial in $x_1,\ldots,x_n$ for $w \in S_n$. It is defined recursively by
$$\partial_k \mathfrak{S}_w = \mathfrak{S}_{ws_k},$$
if $\ell(w)>\ell(ws_k)$, where $\partial_k(f)= { f- s_k(f) \over x_i-x_{i+1}}$. For the longest element $w_0$ in $S_n$, $\mathfrak{S}_w=x_1^{n-1}x_2^{n-2}\ldots x_{n-1}$.\\

In \cite{FK99}, Fomin and Kirillov claimed the following nonnegativity conjecture.
\begin{conj}[Nonnegativity conjecture \cite{FK99}] \label{posfk}
Schubert polynomials $\mathfrak{S}_w(\theta)$ evaluated at Dunkl elements (by replacing $x_i$ by $\theta_i$) can be written as a nonnegative linear combination of noncommutative monomials $[i_1,j_1][i_2,j_2]\cdots [i_m j_m]$.
\end{conj}

Note that the nonnegative formula for $\mathfrak{S}_w(\theta)$ implies the combinatorial formula of the Littlewood-Richardson coefficients for the flag variety for type A from the following theorem:
\begin{thm}\cite{FK99}\label{lr}
For $w,u,v \in S_n$ with $\ell(w)=\ell(u)+\ell(v)$, the Littlewood-Richardson coefficients $p^w_{u,v}$ is the same as the coefficient of $v$ in $w \cdot \mathfrak{S}_u$.
\end{thm}

\begin{rmk}\label{cases}
The nonnegativity conjecture is known for few $w$. Pieri rule, the case $w_{a,b}:=s_a s_{a+1}\ldots s_b$ for $1\leq a\leq b \leq n$, is conjectured in \cite{FK99} and proved by Postnikov \cite{Pos99}. The dual case $w=w_{a,b}^{-1}$ follows from Pieri rule as well. Recently, the nonnegative formula for the hook-shape case $w_{b+1,c}^{-1}w_{a,b}$ for $a\leq b\leq c$ is given by M\'{e}sz\'{a}ros, Panova, and Postnikov \cite{MPP14}. Positivity is also known for hook-shape with one box added at (2,2) in \cite{MPP14}.
\end{rmk}

\section{Application of the nonnegativity conjecture}
Let $\psi: \zsn \rightarrow \mathbb{Z}$ be the ring homomorphism defined by $w\mapsto 1$ for all $w \in S_n$. Then the following is the relation between the Fomin-Kirillov algebra and the CSM classes.

\begin{thm}\label{main}
For $w,v \in S_n$, we have
$$c(w;v)= \psi( w * \mathfrak{S}_v).$$
\end{thm}
Note that assuming Conjecture \ref{posfk}, we have
$$c(w;v)= \psi( w * \mathfrak{S}_v) \geq \psi(w \cdot\mathfrak{S}_v) =\sum_u p^w_{u,v}$$
by Theorem \ref{lr} and the fact that $w \cdot\mathfrak{S}_v$ is the degree $\ell(w)-\ell(v)$ part of $w * \mathfrak{S}_v$. Therefore, Conjecture \ref{posfk} guarantees Conjecture \ref{refined} for type A. Since the term $\sum_u p^w_{u,v}$ is strictly positive when $v\leq w$, Conjecture \ref{poscsm} also follows. In fact, the nonnegativity conjecture is much stronger than positivity of the CSM classes of Schubert cells. Note that it is well-known that for $v \leq w$, there exists $u$ such that $p^w_{u,v}$ is positive. \\

The proof of Theorem \ref{main} will be given in Section 7, after describing the relation between the CSM classes and the generalized nil-Hecke algebra. We discuss applications of Theorem \ref{main} in this section. \\

\begin{example}
For $\ell(w)=\ell(v)+1$, we have
$$c(w;v)= \psi( w * \mathfrak{S}_v)=\psi( w \cdot \mathfrak{S}_v)=\sum_u p^w_{u,v}>0.$$
\end{example}
\begin{example}
Consider $v=s_{i_1}\ldots s_{i_k}$ when $s_{i_j}$ pairwise commutes. Then
$$\mathfrak{S}_v=\prod_{j=1}^k \mathfrak{S}_{s_k} = \prod_{j=1}^k \sum_{a=1}^j \theta_a,$$
which is the same as
$$\prod_{j=1}^k \left(\sum_{x\leq j <y} [xy]\right).$$
This provides a positive formula for $c(w;v)$ for any $w$ and $v=s_{i_1}\ldots s_{i_k}$ such that $s_{i_j}$ pairwise commutes.
\end{example}
\begin{example} \cite{Pos99} For $1\leq a\leq b\leq n$, we have
$$\mathfrak{S}_{w_{a,b}}= \sum [x_1, y_1] \ldots [x_{b-a+1} y_{b-a+1}]$$
where the sum runs over all sequences $x_1,\ldots x_{b-a+1},y_1,\ldots,y_{b-a+1}$ satisfying (i) $x_i\leq b <y_i$; (2) $x_i$ are all distinct; (3) $y_1\leq \cdots \leq y_{b-a+1}$. It provides a positive formula for $c(w;v)$ for any $w$ and $v=w_{a,b}$.
\end{example}

See Remark \ref{cases} for $v \in S_n$ such that the nonnegativity conjecture of $\mathfrak{S}_v(\theta)$ is known. \\

It seems that conjecturally, we have a positive refinement of $c(w;v)$: for $w,v,u \in W$ there are nonnegative integers
$$f(w,v,u)= \text{ coefficient of }u \text{ in } w*\mathfrak{S}_v,$$
such that $f(w,u,v)=p^w_{u,v}$ when $\ell(w)=\ell(u)+\ell(v)$, and $c(w;v)=\sum_u f(w,v,u)$. The nonnegativity suggests that $f(w,v,u)$ are nonnegative. One natural question is whether the constants $f(w,v,u)$ have a meaning in algebraic geometry or reperesentation theory. In the next two sections, we show that $f(w,v,u)$ is a certain sum of structure constants of the cohomology of the Bott-Samelson variety.

\section{Generalized nil-Hecke algebra}
In this section, we discuss nil-Hecke algebra studied by Kostant and Kumar \cite{KK86} and generalized nil-Hecke algebra studied by Berenstein and Richmond \cite{BR15}. \\

\subsection{The nil-Hecke algebra}

Recall that $\alpha_1,\ldots,\alpha_r$ are the positive simple roots of $G$. Let $A=(a_{ij})_{i,j \in I}$, where $I$ is the set $\{1,2,\ldots,r\}$ and $a_{ij}=\langle \alpha_i , \alpha^\vee_j\rangle$. Let $S$ be the symmetric algebra of the root lattice, which is isomorphic to $H^*_T(X)$. Let $F=Frac(S)$ be the fraction field of $S$. There is a Weyl group action on $F$ given by $s_i(\alpha_j)=\alpha_j- a_{ij} \alpha_i$ and $w(q_1 q_2)=w(q_1) w(q_2)$ for $w \in W, q_1,q_2 \in S$. Let $S_W$ be the smash product $F\rtimes \mathbb{Q}[W]$ subject to the relation:
\begin{equation} \label{wq} wq= w(q) \cdot w\end{equation}
for all $q\in S, w\in W$. There is a coproduct structure in $S_W$ given by
\begin{equation}\label{coproduct}
\Delta(qw)= qw \otimes w.
\end{equation}
The coproduct $\Delta: S_W \rightarrow S_W \otimes_S S_W$ is also a homomorphism of algebras. Note that the multiplication on $S_W \otimes_S S_W$ is component-wise product: for $a,b,c,d \in S_W$, we have $(a\otimes b)\cdot(c\otimes d)=ac \otimes bd$. \\

For $i\in I$, define an element $x_i \in S_W$ by $x_i:= \alpha_i^{-1} (s_i-1)$. The nilHecke algebra $\mathbb{A}$ is a subalgebra of $S_W$ generated by $x_i$ for all $i\in I$ and $S$. Let $w \in W$ and let $w=s_{i_1}\ldots s_{i_m}$ be a reduced expression of $w$. Then the element $x_{i_1}\ldots x_{i_m}$ does not depend upon the choice of reduced expression of $w$ by \cite[Prop 4.2]{KK86}. We define $x_w$ as $x_{i_1}\ldots x_{i_m}$. Then $\{x_w\}_{w \in W}$ is a left (as well as right) $S$-basis of $\mathbb{A}$.\\

There is an induced coproduct on $\mathbb{A}$ given by
\begin{equation}\label{coproduct}
\Delta(x_i)= x_i \otimes 1 + s_i \otimes x_i = x_i \otimes 1+ (1+\alpha_i x_i) \otimes x_i.
\end{equation}

It turns out that structure constants of the coproduct is the same as the structure constants of the equivariant cohomology of the flag variety $X=G/B$.

\begin{thm} \cite{KK86} For $w,u,v \in W$, Let $\sigma_w$ be the Schubert basis in $H^*_T(X)$ Poincare dual to $[X(w)]$. Let $p^{w}_{u,v}$ be the (equivariant) Littlewood Richardson coefficients, i.e., constansts satisfying $\sigma_u \sigma_v = \sum_w p^w_{u,v} \sigma_w$. Then we have
$$\Delta(x_w)= \sum_{u,v \in W} p^{w}_{u,v} x_u \otimes x_v.$$
\end{thm}

The follwing lemma will be useful when computing $\Delta(x_w)$.
\begin{lemma}\label{xi} For $i \in I$ and $\lambda \in \mathfrak{h}^*$, we have
$$x_i \lambda = s_i(\lambda) x_i - \langle \lambda, \alpha^\vee \rangle$$
\end{lemma}
\begin{proof} By Equation (\ref{wq}) for $w=s_i$ and $q=\lambda \in \mathfrak{h}^* \subset S$, we have
$$(1+\alpha_ix_i) \lambda= s_i(\lambda) (1+\alpha_i x_i).$$
By using $s_i(\lambda)=\lambda-\langle \lambda, \alpha^\vee \rangle \alpha_i$ and simplyfing, the lemma follows. \end{proof}
\subsection{The generalized nil-Hecke algebra}

Let $\hat W$ be the free monoid generated by $s_1,\ldots, s_r \in W$. There is an surjective homomorphism of monoids $\mu: \hat W \rightarrow W$ sending $s_i$ to $s_i$. The action of $\hat W$ on $S$ is induced by $\mu$. The generalized nil-Hecke algebra $\hata$ is the smash product $S \rtimes \mathbb{Q}[\hat W]$. For $\bfi=({i_1},\ldots, {i_m}) \in I^m$, define $x_\bfi$ by $x_{i_1}\ldots x_{i_m}$. Then the set $\{ x_\bfi \mid \bfi \in I^m \text{ for } m\geq 0 \}$ forms a $S$-basis of $\hata$. Note that $\mu$ induces the surjection $\hata \rightarrow \mathbb{A}$, denoted by the same $\mu$. Let $w_\bfi$ be the element $s_{i_1}\ldots s_{i_m}$ in $W$ and let $\ell(\bfi)$ be $m$ for $\bfi \in I^m$. It is obvious from the definition that $\ell(\bfi)\geq \ell(w_\bfi)$, and equality holds if and only if $s_{i_1}\ldots s_{i_m}$ is a reduced word of $w_\bfi$, and we call $\bfi$ reduced. The surjection $\mu$ sends $x_\bfi$ to $x_{w_\bfi}$ if $\bfi$ is reduced, 0 otherwise.\\

The generalized nil-Hecke algebra also has a coproduct structure induced by (\ref{coproduct}), and the structure constants is the same as the structure constants of the $T$-equivariant cohomology of the Bott-Samelson varieties. We follow the notations in \cite{BR15} to review some results about Bott-Samelson variety. \\

 For $\bfi=(i_1,\ldots,i_m) \in I^m$ and a subset $K=\{m_1<\ldots<m_k\}$ of $[m]$, denote $\bfi_K$ the subsequence $(i_{m_1},\ldots, i_{m_k}) \in I^k$ of $\bfi$. For each $\bfi=(i_1,\ldots,i_m) \in I^m$, the $\bfi$-th Bott-Samelson variety $\Gamma_\bfi=\Gamma_\bfi(G)$ of $G$ is defined by:
$$\Gamma_\bfi = \big( P_{i_1} \times_B P_{i_2} \times_B \ldots \times_B P_{i_m} \big) /B $$
where $P_i, i\in I$ stands for the $i$-th minimal parabolic subgroup. It is well-known that $T$-equivariant cohomology $H^*_T(\Gamma_\bfi)$ has a $S$-basis $\{\sigma^T_K\}$, where $K$ runs over all subsets of $[m]$. The equivariant Bott-Samelson coefficients $p^{\bfi,K}_{K',K''} \in S$ are defined by
$$\sigma^T_{K'} \sigma^T_{K''}= \sum p^{\bfi,K}_{K',K''} \sigma^T_K,$$ 
where the summation runs over all subset $K\subset [m]$ such that $K' \cup K'' \subset K$ and $|K| \leq |K'|+|K''|$. The structure constants of the ordinary cohomology of the Bott-Samelson varieties are $p^{\bfi,K}_{K',K''}$ with $|K|=|K'|+|K''|$ and these constants are integers, but not necessarily nonnegative integers.
\\

The following result was proved by H. Duan \cite{Duan} for the ordinary cohomology and by M. Willems \cite{Wil04} in the equivariant setting.

\begin{thm} \label{mu} Let $G$ be a Kac-Moody group and $W$ be its Weyl group. Then for any sequence $\bfi=(i_1,\ldots,i_m) \in I^m$ one has:
\begin{enumerate}
\item The pullback of the canonical projection $\mu_\bfi : \Gamma_\bfi \rightarrow G/B$ is an algebra homomorphism $\mu^*_\bfi: H^*_T(G/B)\rightarrow H^*_T(\Gamma_\bfi)$ given by
$$\mu^*_\bfi(\sigma^T_w)= \sum_{K \subset [m], \bfi_K \in R(w)} \sigma^T_K$$
for all $w\in W$ (If there is no $K$ satisfying $\bfi_K\in R(w)$, $\mu^*_\bfi(\sigma^T_w)=0$). Here $R(w)$ is the set consisting of all reduced words of $w$.
\item If $\bfi \in R(w)$ for some $w\in W$, then for any $u,v \in W$, the equivariant Littlewood-Richardson and Bott-Samelson coefficients are related by:
$$ p^w_{u,v} = \sum p^{\bfi,[m]}_{K',K''},$$
where the summation runs over all subsets $K',K'' \subset [m]$ such that $\bfi_{K'} \in R(u), \bfi_{K''} \in R(v)$.
\end{enumerate}
\end{thm}

Let $\mu : \hata \rightarrow \mathbb{A}$ be the map sending $A_i$ to $A_i$. In other words, $\mu$ sends $A_\bfi$ to $A_{w_\bfi}$ if $\ell(\bfi)=\ell(w_\bfi)$, and 0 otherwise. One can define the coproduct $\Delta : \hata \rightarrow \hata \otimes_S \hata$ defined by the equation $(4)$ and $\Delta(ab)=\Delta(a)\Delta(b)$. Note that the definition is the same as the coproduct on $\mathbb{A}$, so that we have $\Delta\circ \mu = (\mu\otimes \mu)\circ \Delta$. 
For words $\bfi,\bfi', \bfi''$, let $p^{\bfi}_{\bfi', \bfi''}$ be an element in $S$ satisfying 
\begin{align}\label{delta} \Delta(x_\bfi)= \sum p^{\bfi}_{\bfi', \bfi''} x_{\bfi'} \otimes x_{\bfi''}.\end{align}
 In \cite{BR15}, those elements are called \emph{relative} (equivariant) Littlewood-Richardson coefficients. The relative equivariant Littlewood-Richardon coefficients and the equivariant Bott-Samelson coefficients are related by the following theorem.

\begin{thm} \label{relative} \cite{BR15} For words $\bfi,\bfi', \bfi''$, we have
$$p^{\bfi}_{\bfi', \bfi''}=\sum p^{\bfi,[m]}_{K',K''}$$
where the sum runs over all pairs $K',K'' \subset [m]$ such that $\bfi_{K'}=\bfi',\bfi_{K''}=\bfi''$.
\end{thm}
 
Note that one can compute $p^{\bfi,[m]}_{\bfi', \bfi''}$ by using (\ref{coproduct}).
\begin{thm}\label{compute}
For $K' \subset [m]$ and $j \in [m]$, let $E_{\bfi,K'}(j)$ be $s_{i_j}$ if $j \in K'$ and $A_{i_j}$ otherwise. Then $p^{\bfi}_{\bfi', \bfi''}$ is the coeeficient of $x_{\bfi''}$ in $\sum_{K'}\prod_{j=1}^m E_{\bfi,K'}(j)$,
where the sum runs over all $K'$ satisfying $\bfi_{K'}=\bfi'$.
\end{thm}
\begin{proof} It directly follows from (\ref{coproduct}) and the fact that $\Delta$ is an algebra homomorphism. \end{proof}

\begin{example}
Let $\bfi=(1,2,1)$ and $\bfi' =(1,1)$. so that there is only one $K'=\{1,3\}$ satisfying $\bfi_{K'}=\bfi'$. Consider $\prod_{j=1}^3 E_{\bfi,K'}(j)=s_1A_2s_1$.\\
\begin{align*}
s_1A_2s_1&=(1+\alpha_1 A_1) A_2 (1+\alpha_1 A_1)\\
&=A_2+\alpha_1 A_{1,2}+ A_2 \alpha_1 A_1+ \alpha_1 A_{1,2} \alpha_1 A_1 \\
&=  A_2+\alpha_1 A_{1,2}+A_1+ (\alpha_1+\alpha_2) A_{2,1}+ \alpha_1 A_{1,1} - \alpha_1 A_{2,1} + \alpha_1\alpha_2 A_{1,2,1}\\
&= A_2+\alpha_1 A_{1,2}+A_1+\alpha_2 A_{2,1}+ \alpha_1 A_{1,1}+ \alpha_1\alpha_2 A_{1,2,1}
\end{align*}
Note that we frequently used Lemma \ref{xi}. The above computation shows that, for example, $p^{121}_{11,21}=\alpha_2$ by Theorem \ref{compute}.

\end{example}
\section{Relationship between the CSM classes and the equivariant Bott-Samelson coefficients}
The goal of this section is to relate $c_T(w;v)$ with the equivariant Bott-Samelson coefficients $p^{\bfi,[m]}_{K',K''}$. \\

Let $\phi$ the isomorphism between the equivariant homology of the flag variety $H^T_*(X)$ and the nilHecke algebra $\mathbb{A}$ as $S$-modules sending $[X_w]^T$ to $x_i \in \mathbb{A}$. Then the Weyl group action and the divided difference operator $\partial_k$ can be purely written in the nilHecke algebra.

\begin{lem} \label{translate} For $k \in I$ and $\xi \in H^T_*(X)$, we have
\begin{align*}
\phi( \partial_k \xi )&= \phi(\xi) x_k,\\
\phi( s_k \xi )&= -\phi(\xi)s_k,\\
\phi( T_k \xi )&= \phi(\xi)(x_k+s_k).
\end{align*}
In the last two equations, the $s_k$ on the left-hand side is the Weyl group action on $H^T_*(X)$, but the $s_k$'s on the right-hand side are an element in $\mathbb{A}$.
\end{lem}
\begin{proof} It is enough to show for $\xi=[X_w]^T$. Then the first equation in the lemma is obvious, and the proof of the second equiation can be checked by comparing the Weyl group action and $\partial_k$ used in Allufi and Mihalcea's paper \cite{AM15} and Kumar's book \cite[(11.1.2)]{Kum02}. The last equation follows by the first two equations and Theorem \ref{rec}.
\end{proof}\\

 By Theorem \ref{rec} and Lemma \ref{translate}, we have the following result. 
\begin{thm}\label{cwv}
Fix a reduced word $s_{i_1}\ldots s_{i_m}$ of $w$. Then $c_T(w;v)$'s are the coefficients of $x_v$ in the product $\prod (s_{i_j}+x_{i_j}) \in \mathbb{A}$:
$$\prod_{i=1}^m (s_{i_j}+x_{i_j})= \sum_v c_T(w;v) x_v$$
where the notation $\prod_{i=1}^m a_i$ means $a_1 a_2 \ldots a_m$. 
\end{thm}

Fix a word $\bfi=(i_1,\ldots,i_m) \in R(w)$ and consider the same product $\prod (s_{i_j}+x_{i_j})$ as an element in $\hata$. Define the map $pr: \hata \otimes_S \hata \rightarrow \hata$ defined by $pr(\sum f_{\bfi,\bfi'} x_\bfi \otimes x_{\bfi'})= \sum f_{ij}x_\bfi$. Note that the map $pr$ is an algbra homomorphism. Indeed, we have 
\begin{align*} pr\big((\sum_{\bfi,\bfi'} f_{\bfi,\bfi'} x_\bfi \otimes x_{\bfi'})(\sum_{\bfj,\bfj'} g_{\bfj,\bfj'} x_\bfj \otimes x_{\bfj'})\big)&= \sum_{\bfi,\bfi',\bfj,\bfj'} f_{\bfi,\bfi'}g_{\bfj,\bfj'} x_{\bfi \cup \bfj}\\&= pr\big(\sum_{\bfi,\bfi',\bfj,\bfj'}  f_{\bfi,\bfi'}g_{\bfj,\bfj'} x_{\bfi \cup \bfj}\otimes x_{\bfi' \cup \bfj'}\big),\end{align*} where $\bfi \cup \bfj$ is the word joining two words $\bfi, \bfj$. \\

Since $\Delta(x_i)= x_i \otimes 1 + s_i \otimes x_i$, we have $pr\circ \Delta(x_i)=x_i+s_i$. Therefore, we have the identity
\begin{align} \label{sx}\prod (s_{i_j}+x_{i_j})= pr\circ \Delta(x_\bfi)\end{align}
in $\hata$. Now we are ready to prove the following theorem.

\begin{thm} \label{ct} For $w \in W$ with $\ell(w)=m$, fix $\bfi \in R(w)$. Then for $v\in W$,
$$c_T(w;v)= \sum_{K,K' \subset [m], \bfi_K \in R(v)} p^{\bfi,[m]}_{K,K'}.$$
\end{thm}
\begin{proof}
By Theorem \ref{cwv} and (\ref{sx}), we have
$$\sum_v c_T(w;v) x_v= \mu\circ pr \circ \Delta(x_\bfi).$$
By (\ref{delta}), we have 
$$\mu\circ pr \circ \Delta(x_\bfi)= \mu (\sum_{\bfi',\bfi''} p^{\bfi}_{\bfi',\bfi''} x_{\bfi'}).$$
Therefore, $c(w,v)$ is the same as $\sum p^{\bfi}_{\bfi',\bfi''}$ where the sun runs over $\bfi' \in R(v)$ and any words $\bfi''$. By Theorem \ref{relative}, the theorem follows.
\end{proof}

\begin{rmk}Note that the same map $pr$ is not an algebra homomorphism if we used the nil-Hecke algebra instead of $\hata$. For instance, $(1\otimes x_1)(1\otimes x_1)$ is 0 in the usual nil-Hecke algebra so that $$0=pr((1\otimes x_1)(1\otimes x_1)) \neq pr(1\otimes x_1)\cdot pr(1\otimes x_1)=1.$$
In $\hata$, they are all equal to 1.
\end{rmk}

\section{Relation with Fomin-Kirillov algebra.}
In this section, we define a certain algebra using Fomin-Kirillov algebra $\fk$ and $\mathbb{Q}[\hat W]$ containing the generalized nilHecke algebra for type $A$. Let $r=n-1$, i.e., $\hat W$ is the free monoid generated by $s_1,\ldots, s_{n-1}$. We relate this with the extended Bruhat action on $\fk$ to prove a stronger version of Theorem \ref{main}. Note that we will be working on type $A$ non-equivariantly and $W$ is the symmetric group $S_n$.\\

Let $\mathcal C$ be the subalgebra of the Fomin-Kirillov algebra generated by Dunkl elements $\theta_i$ for $i=1,\ldots,n-1$. Let $\mathcal D$ be the subalgebra of $\mathcal C$ generated by $\theta_i-\theta_{i+1}$ for $i=1,\ldots,n-1$.
Let $\mathcal B$ be the smash product $\fk \rtimes \mathbb{Q}[\hat W]$ subject to the relation:
\begin{equation} \label{kl} x_i [kl]= s_i([kl]) x_i + \delta'_{(i,i+1),(k,l)},\end{equation}
where $\delta'_{(i,i+1),(k,l)}$ is 1 if $i=k,i+1=l$, -1 if $i=l,i+1=k$, 0 otherwise. \\

Fomin and Kirillov \cite{FK99} defined an operator $\Delta_{ij}$ on $\fk$ defined by $\Delta_{ij}([kl])=\delta'_{(i,j),(k,l)}$ and
\begin{align*} 
\Delta_{ij}([x_1y_1][x_2y_2] \ldots [x_m y_m])&= \Delta_{ij}([x_1 y_1]) [x_2y_2] \ldots [x_m y_m]\\
&+ s_{ij}([x_1y_1]) \Delta_{ij}([x_2 y_2]) [x_3 y_3] \ldots [x_m y_m]\\
&+ \ldots\\
&+ s_{ij} ([x_1y_1] \ldots [x_{k-1} y_{k-1}]) \Delta_{ij}([x_m y_m]).
\end{align*}
By Equation (\ref{kl}), for $k \in I$ and an element $p$ in $\fk$ we have
\begin{equation}\label{xkp} x_k p= s_k(p) x_k + \Delta_{k,k+1}(p).\end{equation}
Indeed, it is enough to show the identity for $p=[x_1y_1][x_2y_2] \ldots [x_m y_m]$ and in this case the identity follows from the definition.\\

It is proved by Fomin and Kirillov that the action $\Delta_{k,k+1}$ is compatible with the divided difference operator $\partial_k$ acting on $H^*(X)$.
\begin{thm} \cite{FK99}
Let $i_{\mathcal C}$ be the isomorphism between $\mathcal C$ and $H^*(X)$. For $1\leq k \leq n$ and $w \in S_n$, we have
$$ i_{\mathcal C} (\partial_k \mathfrak{S}_w(\theta) )= \Delta_{k,k+1}\circ i_{\mathcal C}( \mathfrak{S}_w(\theta)).$$
\end{thm}

Define an algebra map $d: \mathbb{Q}[\hat W] \rightarrow \mathbb{Q}[W]$ defined by
$$d(x_\bfi)= w_\bfi.$$

The algebra $\mathcal B$ and the extended Bruhat action on the Fomin-Kirillov algebra is related by the following lemma.
\begin{lemma}\label{dx}
Let $\bfi$ be a word with $w_\bfi=w \in S_n$. Consider the expansion
$$ x_\bfi [kl] - w_\bfi ([kl]) x_\bfi = \sum_{\bfi'} c_{\bfi'}x_{\bfi'}$$
by successively applying (\ref{kl}) to $x_\bfi [kl]$, where $c_{\bfi'} \in \mathbb{Z}$ and $k<l$. Then 
$$d( \sum_{\bfi'} c_{\bfi'}x_{\bfi'})=  \begin{cases} w_\bfi t_{kl}, & \text{if } \ell(w_\bfi t_{kl})<\ell(w_\bfi) \\ 0, & \text{otherwise.} \end{cases}$$
\end{lemma}
\begin{proof}
Let $\bfi = (i_1,\ldots,i_m)$ and let $t_j$ an element in $S_n$ defined by $s_{i_{j+1}}s_{i_{j+2}}\ldots s_{i_m}$. By using (\ref{kl}), we have
\begin{align*} 
 x_\bfi [kl] - w([kl]) x_\bfi = \sum_{j=1}^m \delta'_{(i_j,i_j+1), (t_j(k),t_j(l))} x_{i_1}\ldots x_{i_{j-1}}\widehat{x_{i_j}}x_{i_{j+1}} \ldots x_{i_m},
\end{align*}
where $j$-th summand omits $x_{i_j}$. Denote $\bfi'_j$ by $x_{i_1}\ldots x_{i_{j-1}}\hat{x_{i_j}}x_{i_{j+1}} \ldots x_{i_m}$.\\

Note that $\delta'_{(i_j,i_j+1), (t_j(k),t_j(l))}$ is nonzero if and only if $w_{\bfi'_j}=wt_{kl}$ since we have $s_{i_j}t_j=t_j t_{t_j(i_j),t_j(i_j +1)}$. Let $V$ denote the set of all $j$ such that $\delta'_{(i_j,i_j+1), (t_j(k),t_j(l))}$ is nonzero. One can show inductively that $\delta'_{(i_j,i_j+1), (t_j(k),t_j(l))}=1$ (resp. $=-1$) if and only if there are even (resp. odd) numbers in $V$ greater than $j$. Therefore, $d(x_\bfi [kl] - w([kl]))= w_\bfi t_{kl}$ if the cardinarity of $V$ is odd, 0 otherwise. Whether $|V|$ is odd or even can be checked by considering $w([kl])$. Indeed, $w([kl])=[ab]$ (resp. $=-[ab]$) for some $a<b$ if and only if $|V|$ is even (resp. odd). Since it is well-known that $w(k)>w(l)$ if and only if $\ell(w)>\ell(wt_{kl})$, we are done (Such pair $(k,l)$ is called an inversion of $w$. See \cite{BB05} for instance).
\end{proof}\\

Since we are working non-equivariantly, define the map $ev: \fk \rightarrow \mathbb{Q}$ by sending all $[ij]$ to zero, and 1 to 1. This induces a map $ev: \fk \rtimes \mathbb{Q}[W]\rightarrow  \mathbb{Q}[ {W}]$. Then Lemma \ref{dx} implies that 
\begin{equation} \label{dev1} ev(d(x_\bfi [kl]))= w* [kl]. \end{equation}

\medskip

We consider subalgebras of $\mathcal B$, $\mathcal C \rtimes \mathbb{Q}[\hat W]$ and $\mathcal D \rtimes \mathbb{Q}[\hat W]$. From (\ref{kl}), we have
\begin{equation}\label{theta} x_i \theta_j = \begin{cases} \theta_j x_i & \text{ if } j\neq i,i+1 \\ \theta_{i+1} x_i + 1 & \text{ if } j=i \\ \theta_i x_i - 1 & \text{ if} j=i+1 \end{cases}\end{equation}
By (\ref{theta}), Any element in $\mathcal C \rtimes \mathbb{Q}[\hat W]$ can be written as $\sum_\bfi p_\bfi(\theta) x_\bfi$ where $p_\bfi(\theta)$ is a polynomial in Dunkl elements $\theta_i$. Similarly, any element in $\mathcal D \rtimes \mathbb{Q}[\hat W]$ can be written as $\sum_\bfi q_\bfi x_\bfi$ where $q_\bfi$ is a polynomial in $\theta_i-\theta_{i+1}$ for $1\leq i <n$.\\

Let $g$ be an surjective algebra map from $S$ to $\fk$ defined by $g(\alpha_i)=\theta_{i+1}-\theta_{i}$ so that the image of $g$ is $\mathcal D$. Recall that the symmetric functions in $\theta_i$ for $1\leq i \leq n$ is 0 in the Fomin-Kirillov algebra so that $g$ is not injective. The map $g$ induces a map from the generalized nil-Hecke ring $S \rtimes \mathbb{Q}[\hat W]$ to $\fk \rtimes \mathbb{Q}[\hat W]$, denoted by $g$ by abusing the notation. Indeed, it is enough to show that relations defining two smash products are compatible with the map $g$: $x_\bfi g(\alpha_i) = g(x_\bfi \alpha_i)$ for a word $\bfi$.\\

For a root $\lambda$ in $S$, we have by definition $x_i \lambda = s_i(\lambda) x_i$ as an identity in $S \rtimes \mathbb{Q}[\hat W]$. By using $s_i= 1+\alpha_i x_i$, we have
$$x_i \lambda = {s_i(\lambda)-\lambda \over \alpha_i}+ s_i(\lambda)x_i=-\langle \lambda, \alpha_i^\vee \rangle+ s_i(\lambda)x_i.$$
By setting $\lambda =\alpha_i$ and comparing with (\ref{theta}), we get the desired equality $x_\bfi g(\alpha_i) = g(x_\bfi \alpha_i)$.\\

The following theorem shows the relationship between the Fomin-Kirillov algebra and generalized nil-Hecke algebra.\

\begin{thm}\label{compute}
For $K' \subset [m]$ and $j \in [m]$, recall that $E_{\bfi,K'}(j)$ is $s_{i_j}$ if $j \in K'$ and $A_{i_j}$ otherwise. Then for a word $\bfi$ and $v\in W$, we have
$$x_\bfi \mathfrak {S}_v(\theta)- \mathfrak {S}_v(\theta) x_\bfi= g(\sum_{K'} \prod_{j=1}^m E_{\bfi,K'}(j))$$
where the sum runs over all $K'$ with $i_{K'}\in R(u)$.

\end{thm}
\begin{proof} We use an induction on $\ell(\bfi)$. Let $\bfi= \bfi' \cup \{k\}$, where $k$ is the last element in $\bfi$. Then one can write 
$$x_\bfi \mathfrak {S}_v(\theta)=x_{\bfi'} x_i \mathfrak {S}_v(\theta).$$
By Equation (\ref{xkp}), we have
\begin{align*}
x_{\bfi'} x_k \mathfrak {S}_v(\theta)&= x_{\bfi'} s_k(\mathfrak {S}_v(\theta))+ x_{\bfi'}(\partial_k\mathfrak {S}_v(\theta))\\
&=x_{\bfi'} \mathfrak {S}_v(\theta) x_k + x_{\bfi'}(\partial_k\mathfrak {S}_v(\theta))(1+ (\theta_{k+1}-\theta_k) x_k)\\
&=x_{\bfi'} \mathfrak {S}_v(\theta) x_k + x_{\bfi'}(\partial_k\mathfrak {S}_v(\theta))(1+ g(\alpha_i) x_k).
\end{align*}
From the first line to the second line, we used the identity $f-s_k(f)= \partial_k(f) (\theta_k-\theta_{k+1})$ for a polynomial $f$ in $\theta_1,\ldots,\theta_n$ by the definition of $\partial_k$. By applying the induction hypothesis on $x_{\bfi'} \mathfrak {S}_v(\theta)$ and $ x_{\bfi'}(\partial_k\mathfrak {S}_v(\theta))$, the theorem follows.

\end{proof}\\

By Theorem \ref{compute}, we have the following corollary.

\begin{cor}\label{gx2} For a word $\bfi$ and $v\in W$, we have

$$x_\bfi \mathfrak {S}_v(\theta)- \mathfrak {S}_v(\theta) x_\bfi= g(\sum p^\bfi_{\bfi',\bfi''}x_{\bfi''})$$
where the sum runs over all $\bfi',\bfi''$ such that $\bfi' \in R(v)$.
\end{cor}

By applying $ev\circ d$ on Theorem \ref{gx2}, we have
\begin{align*}
w*\mathfrak {S}_v(\theta)&= ev(d(x_\bfi\mathfrak {S}_v(\theta)) \quad \text{(By Equation (\ref{dev1}))}\\
&= ev(d(x_\bfi\mathfrak {S}_v(\theta)-\mathfrak {S}_v(\theta) x_\bfi)\\
&= ev(d(g(\sum_{\substack{\bfi',\bfi'' \\ \bfi' \in R(v)}} p^\bfi_{\bfi',\bfi''}x_{\bfi''}))).
\end{align*}
Note that the composition $ev\circ d \circ g$ is the map from $\hata= S\rtimes \mathbb{Q}[\hat W]$ to $\mathbb{Q}[W]$, sending $\sum_\bfi  f_\bfi x_\bfi$ to $\sum_{\substack{ \bfi \text{ reduced} }} i(f_\bfi) w_\bfi$ where $i$ is the projection map from $S$ to $\mathbb{Q}$. Note that $p^\bfi_{\bfi',\bfi''}$ is a homogeneous polynomial in $\alpha_i$'s of degree $\ell(\bfi')+\ell(\bfi'')-m$ so that $p^\bfi_{\bfi',\bfi''}$ is an integer if and only if $\ell(\bfi')+\ell(\bfi'')=m$ (See \cite{BR15,Wil04} for instance). Therefore, we proved the following theorem.

\begin{thm} For $w,v \in W$, we have
$$w*\mathfrak {S}_v(\theta)=\sum_{\substack{\bfi',\bfi'' \\ \bfi' \in R(v), w_{\bfi''}=u\\\ell(\bfi')+\ell(\bfi'')=m}} p^{\bfi}_{\bfi',\bfi''} u.$$
\end{thm}
As a corollary, we obtain the following interpretation of $f(w,v,u)$.
\begin{cor}\label{fwvu}
For $w \in W$ with $\ell(w)=m$, fix $\bfi \in R(w)$. Then for $u,v\in W$, we have
$$f(w,v,u)= \sum_{\substack{K,K' \subset [m] \\ |K|+|K'|=m\\ \bfi_K \in R(v), w_{\bfi_K'} = u} } p^{\bfi,[m]}_{K,K'}.$$
\end{cor}
We end the section by proving Theorem \ref{main}.\\
{\it Proof of Theorem \ref{main}.} Recall that by Theorem \ref{ct}, for $w,v \in W$ with $\ell(w)=m$ with a fixed $\bfi \in R(w)$, we have
$$c_T(w;v)= \sum_{K,K' \subset [m], \bfi_K \in R(v)} p^{\bfi,[m]}_{K,K'}.$$
By applying the map $i:S\rightarrow \mathbb{Q}$, we have
$$c(w;v)= \sum_{\substack{K,K' \subset [m], \bfi_K \in R(v)\\|K|+|K'|=m}} p^{\bfi,[m]}_{K,K'}.$$
Since $c(w,v)=\sum_u f(w,v,u)$ by the definition of $f(w,v,u)$, Theorem follows by Corollary \ref{fwvu}. \qed

\section{Concluding remark}
Although we provided an intepretation of $f(w,v,u)$ in terms of equivariant Bott-Samelson constants, it is not clear why $f(w,v,u)$ or $c(w;v)$ is nonnegative. Note that the Littlewood-Richardson coefficients $p^w_{u,v}$ follows from Kleinman's transversality theorem \cite{Kle74}, and the equivariant version was established by Graham \cite{Gra01}. It would be very interesting to relate the work in this paper and Huh's proof \cite{Huh} of the positivity of $c(w,v)$ when $w,v$ are elements in the $W/W_P$, where $G/P$ is the Grassmannian.\\

It would be also interesting to look for the representation theoritic meaning of $c_T(w;v)$. For $w,v \in W$, the degree of $c_T(w;v)$ in $\alpha_i$'s is $\ell(v)$. Denote the homogeneous polynomial of degree $\ell(v)$ in $c_T(w;v)$ by $a(w,v)$. We prove the following theorem in this section.

\begin{thm} \label{top} For $w,v \in W$, $a(w,v)$ is the same as the localization of the Schubert class $\sigma^w$ at the point $v$, denoted by $\sigma^w(v)$. \end{thm}
$\sigma^w(v)$ has a representation meaning by Kumar \cite{Kum95} and it has a positive formula given by Billey \cite{Bil99}. It would be interesting to extend Kumar's work \cite{Kum95} to find a reperesentational meaning of $c_T(w;v)$.\\

{\it Proof of Theorem \ref{top}.} Introduce the $\mathbb{Z}$-degree $z$ on $\mathbb{A}$ by $z(x_i)=-1$ and $z(\alpha_i)=1$. Note that the degree map $z$ is well-defined since the equation in Lemma \ref{xi} is homogeneous with respect to the degree map $z$. Then the degree $0$ component of $\sum_v c_T(w,v) x_v $ with respect to $z$ is $\sum_v a(w,v) x_v$. On the other hand, for $w=s_{i_1}\ldots s_{i_m}$, the sum $\sum_v a(w,v) x_v$ is the degree $0$ part of the product $\prod_{i=1}^m (s_{i_j}+x_{i_j})$ where $w=s_{i_1}\ldots s_{i_m}$ by Theorem \ref{cwv}. Since $z(s_i)=0$ and $z(x_i)=-1$, the degree $0$ part is $\prod_{i=1}^m s_{i_j}=w$ in $\mathbb{A}$. Since it is well-known that $w= \sum_v \sigma^w(v) x_v$ (See \cite{Kum02} for example), we are done.\qed

\section*{acknowledgment}
I thank Leonardo Mihalcea for helpful discussions.

\end{document}